\documentclass[twoside,a4paper,12pt,centertags]{amsart}
\usepackage{amsmath,amssymb,verbatim,vmargin}
\usepackage[bookmarks=true]{hyperref}   

\theoremstyle{plain}
\newtheorem{thm}{Theorem}[section]
\newtheorem{lem}[thm]{Lemma}
\newtheorem{cor}[thm]{Corollary}
\newtheorem{prop}[thm]{Proposition}

\theoremstyle{definition}
\newtheorem{defn}[thm]{Definition}
\newtheorem{rem}[thm]{Remark}

\title[On a quadratic estimate]
{On a quadratic estimate related to  the Kato conjecture and boundary value problems}
\author{Pascal Auscher} \author{Andreas Axelsson} \author{Alan M{\lowercase {c}}Intosh}
\address{Pascal Auscher, Universit\'e de Paris-Sud, UMR du CNRS 8628, 91405 Orsay Cedex, France}
\email{pascal.auscher@math.u-psud.fr}
\address{Andreas Axelsson, Matematiska institutionen, Stockholms universitet, 106 91 Stockholm, Sweden}
\email{andax@math.su.se}
\address{Alan McIntosh, Centre for Mathematics and its Applications, Mathematical Sciences Institute, Australian National University, Canberra ACT 0200, Australia}
 \email{alan.mcintosh@anu.edu.au}

\newcommand{\qedend}{ }

\mathchardef\semic="303B
\newcommand{\wedg}{\mathbin{\scriptstyle{\wedge}}}

\newcommand{\R}{{\mathbf R}}
\newcommand{\C}{{\mathbf C}}
\newcommand{\Z}{{\mathbf Z}}
\newcommand{\N}{{\mathbf N}}

\newcommand{\mH}{{\mathcal H}}
\newcommand{\mC}{{\mathcal C}}

\newcommand{\mL}{{\mathcal L}}

\newcommand{\mD}{{\mathcal D}}

\DeclareMathOperator{\re}{Re}

\newcommand{\brac}[1]{\langle #1 \rangle}
\newcommand{\supp}{\text{{\rm supp}}\,}
\newcommand{\dist}{\text{{\rm dist}}\,}

\newcommand{\nul}{\textsf{N}}
\newcommand{\ran}{\textsf{R}}
\newcommand{\dom}{\textsf{D}}

\newcommand{\clos}[1]{\overline{#1}}

\newcommand{\dyadic}{\triangle}
\newcommand{\sgn}{\text{{\rm sgn}}}
\newcommand{\barint}{\mbox{$ave \int$}}
\newcommand{\divv}{{\text{{\rm div}}}}
\newcommand{\curl}{{\text{{\rm curl}}}}

\newcommand{\ud}{\underline{d}}
\newcommand{\hut}[1]{\check #1}

\newcommand{\ta}{{\scriptscriptstyle \parallel}}
\newcommand{\no}{{\scriptscriptstyle\perp}}

\newcommand{\oB}{{\overline B}}
\newcommand{\uB}{{\underline B}}
\newcommand{\qe}[1]{\int_0^\infty\|#1\|^2\,\frac{dt}t}
\newcommand{\qeinf}[2]{\int_{#2}^\infty\|#1\|^2\,\frac{dt}t}
\newcommand{\qezero}[2]{\int^{#2}_0\|#1\|^2\,\frac{dt}t}

\def\barint_#1{\mathchoice
            {\mathop{\vrule width 6pt
height 3 pt depth -2.5pt
                    \kern -8.8pt
\intop}\nolimits_{#1}}%
            {\mathop{\vrule width 5pt height
3 pt depth -2.6pt
                    \kern -6.5pt
\intop}\nolimits_{#1}}%
            {\mathop{\vrule width 5pt height
3 pt depth -2.6pt
                    \kern -6pt
\intop}\nolimits_{#1}}%
            {\mathop{\vrule width 5pt height
3 pt depth -2.6pt
          \kern -6pt \intop}\nolimits_{#1}}}
          
\DeclareMathOperator{\I}{I}

\usepackage{color}

\definecolor{gr}{rgb}   {0.,   0.8,   0. } 
\definecolor{bl}{rgb}   {0.,   0.5,   1. } 
\definecolor{mg}{rgb}   {0.7,  0.,    0.7}

\begin{document}

\begin{abstract} We provide a direct proof of a quadratic estimate that plays a central role in the determination of domains of square roots of elliptic operators and, as   shown  more recently, in some boundary value problems with $L^2$ boundary data.   We develop the application to the Kato conjecture and to a Neumann problem. This quadratic estimate enjoys some equivalent forms in various settings.  This  gives new results in the  functional calculus of Dirac type operators on forms.  
\end{abstract}
\maketitle

\subjclass{MSC classes: 35J25, 35J55, 47N20, 47F05, 42B25}

\subjclass{Keywords: Littlewood-Paley estimate, functional calculus, boundary value problems, second order elliptic equations and systems, square root problem}

\section{Introduction}

The  goal of this paper is  first  to present  a self-contained and simple proof of the following quadratic estimate,  and second,  to convince the reader that this is a central estimate in this area.

\begin{thm}\label{th:main}  Let $n,m$ be positive integers, $\mH= L^2(\R^n , \C^m)$ and  $D, B$ be operators  on $\mH$ satisfying the requirements ${\mathrm(H)}$. Then one has the quadratic estimate
\begin{equation}  \label{eq:sfBD}
  \int_0^\infty\| t^kBD(I+t^{2k}BDBD)^{-1} u \|^2 \, \frac{dt}t \lesssim \|u\|^2, \qquad  \text{for all }\ u\in \mH.
\end{equation}
 \end{thm}
 
One uses $(\, , \, )$ and $\|\, \|$ for the hermitian product and norm on $\mH$. The hypotheses ({\textrm H}) consist of the following set of requirements. 
 \begin{itemize}
 \item[(H1)] The operator $D:\dom(D) \longrightarrow \mH$ is a  homogeneous $k$th order differential operator with constant coefficients. 
 \item[(H2)]  $D$ is self-adjoint.
\item[(H3)]  $D$ is strictly accretive on its range, i.e.
$$
 \|\nabla^k  u\| \lesssim \|Du\|  , \quad  \text{for all }\ u \in \dom(D) \cap \clos{\ran(D)}.
$$
\item[(H4)]  $B$ is a bounded  operator on $\mH$.
\item[(H5)] $B$ is strictly accretive on $\ran(D)$:  there is a constant $\delta>0$ such that 
$$
\re ( B Du, Du )\ \ge  \delta  \|Du\|^2, \quad   \text{for all } \ u \in \dom(D).
$$
\item[(H6)] (Off-diagonal decay)
 For every integer $N$ there exists $C_N>0$
such that
\begin{equation} \label{odn}
\|t^kBD(I+t^{2k}BDBD)^{-1} u\|_{L^2(E)} \le C_N \brac{\dist (E,F)/t}^{-N}\|u\|
\end{equation}
for all $t>0$, 
whenever $E,F \subset \R^n$ are closed sets,  $u \in \mH$
satisfies $\supp u\subset F$.  We have set $\brac x:=1+|x|$, and
$\dist(E,F) :=\inf\{|x-y|:x\in E,y\in F\}$.

\end{itemize}

\medskip

In this paper, if $A$ is a (densely defined) unbounded linear operator on $\mH$ then $\dom(A), \nul(A), \ran(A)$ denote respectively, its domain, null space and range.  In (H3),  $\nabla^k  u=(\partial^\alpha u_j)_{|\alpha|=k,1\leq j\leq m}$ consists of all the partial derivatives of $u$ of order $k$. 
 The assumptions (H2,4,5)  imply that $BD$ has  spectrum contained in a double sector of the complex plane centered around $\R$ and   give boundedness of the operator in (H6) (See Proposition \ref{prop:typeomega}).    The constant  in \eqref{eq:sfBD}  depends on the implicit constants in ({\textrm H}).  

 We mention right away that  our interest is in operators $B$  of multiplication by  $B(x)$, identified as a matrix having coefficients in $L^\infty(\R^n, \C)$ , in which case (H5) is a form of G\aa rding inequality. When $D$ is first order, i.e.~$k=1$, and $B$ is such a multiplication operator, then the off-diagonal decay (H6) holds true. Moreover, when $k>1$, then (H6) is still satisfied in the case of most interest to us. (See Section \ref{offdiag}.)  However, we wanted to enlighten the observation that only (H6) is needed (in our arguments).   We also stress that $D$ is not assumed to be one-to-one.

This theorem is proved in \cite{AKMc}  for first order $D$, i.e. $k=1$, as  a corollary of another quadratic estimate. Our direct proof is shorter and simpler from the algebraic point of view, and also from the  analysis point of view even though the same deep ideas are involved (Carleson measures, $T(b)$ argument).  Furthermore,  our proof   allows a simultaneous treatment  of higher order $D$, i.e. $k\ge 2$, which is new.

 The  interest of proving a quadratic estimate  is mainly in  the following proposition as a corollary of  results developed in \cite{Mc}.

\begin{prop}\label{cor:main} Assume that $B,D$ satisfy (H2,4,5) on a Hilbert space $\mH$, that $BD$ satisfies the quadratic estimate \eqref{eq:sfBD}, and that $B^*D$ satisfies the same quadratic estimate with $B^*$ in place of $B$. Then the operator $\sgn(BD)$ is bounded  on $\mH$ and  invertible on  $\clos{\ran(BD)}$. 
\end{prop}

The operator $\sgn(BD)$ is zero on $\nul(BD)$ and satisfies $(BDBD)^{1/2}=\sgn(BD)BD$ on $\dom(D)$.
More is true, in particular $BD$ has a bounded holomorphic functional calculus on $\mH$. We remark that the specific nature of   $\mH$, $B$ and $D$ is not used in this proposition, which follows from operator theoretic considerations, once quadratic estimates for the operators  $BD$ and $B^*D$   have been proved.

When $k=1$, we obtain the following corollary to this result, once we have proved Proposition \ref{pseudoloc}. Note that if $B$ satisfies (H4,5), then so does $B^*$.

\begin{cor} Assume that $B,D$ satisfy (H1-5) on  $L^2(\R^n,\C^m)$, 
that $k=1$, 
and that $B$ is multiplication by a function $B\in L^\infty(\R^n,\mL(\C^m))$.
 Then the operator $\sgn(BD)$ is bounded  on $\mH$ and  invertible on  $\clos{\ran(BD)}$. 
\label{prop:firstorder}\end{cor}

 When $k>1$, we do not know if $B$ being a multiplication operator is enough, in addition to (H1-5), to conclude for the boundedness of $\sgn(BD)$. It is the case when $B,D$ are as in Section \ref{sec:K} and the boundedness of $\sgn(BD)$ appears new.

Known consequences  of the boundedness of  operators $\sgn(BD)$ are  short proofs of the Kato conjecture for  elliptic systems in divergence form \cite{AHLMcT, AHMcT}  and  the boundedness of the Cauchy integral on Lipschitz curves \cite{CMcM} (see Section \ref{sec:KC}).
 In Section \ref{ESSIBD}  we give a pedestrian account of one of the results obtained in  \cite{AAM} concerning boundary value problems for second order elliptic systems    which, in particular,   give new proofs of solvability for single equations with real symmetric coefficients established in \cite{JK1, DJK, KP}.   In Section \ref{sec:relops}, we show that the quadratic estimate for $BD$ has different  equivalent formulations with operators built by functional analytic considerations, including the one studied in \cite{AKMc}, and  we present a new application related to BVPs for differential forms.
 
  The quadratic estimate (\ref{eq:sfBD})  has some further interest. It is easily seen to be stable under perturbation with lower order terms. This implies a simple proof of the Kato conjecture for inhomogeneous    elliptic operators (or systems) in divergence form, where previously it required an interpolation procedure from pure  operators or a longer argument  \cite{AT} (See also \cite{AKMc1}).  The extension to inhomogeneous situations is motivated also by potential applicability to time-harmonic  Maxwell's equations. See the introduction of \cite{AAM}.

\medskip
 
\noindent{\bf Acknowledgments.}

 This work grew out from a visit of the last two named authors to the Universit\'e Paris-Sud. A.A. and A.Mc. thank this University  for partial support. This research was also supported by the Australian Government through the Australian Research  Council and through the International Science Linkages FAST program. 

 Thanks are also due to the organisers of the El Escorial 2008 conference for opportunity of presenting this work both in lectures and in these proceedings, and for a well organised and stimulating conference. We also thank the anonymous referee for suggestions that improved the presentation of this article.

\section{Kato and Cauchy}\label{sec:KC}

 We present two typical applications of the boundedness of $\sgn(BD)$ already in the literature (at least when $k=1$). We refer to \cite{AKMc} and the references therein for a number of further applications.  

\subsection{Kato}\label{sec:K}
 
 The application to the square root of elliptic systems  $L=  (\nabla^{k})^* A \nabla^k$ is as follows:   $A$ is multiplication by a bounded  matrix $A(x)$, and one assumes the G\aa rding inequality
$$
\re ( A \nabla^k  u   , \nabla^k  u ) \ge  \delta \|\nabla^k  u\|^2, \quad  \text{for all } \ u \in H^k(\R^n, \C^N).
$$
Here $u$ is $\C^N$-valued.
 Thus, we set $L^2(\R^n, \C^m) = L^2(\R^n, \C^N) \oplus L^2(\R^n, \C^{Np})$ where $m=N+Np$ and $p$ is the length of the array $\nabla^k$,
  \begin{equation}  \label{eq:defhigherD}
 D:=
     \begin{bmatrix}
        0 &  (\nabla^{k} )^* \\
        \nabla^k  & 0
     \end{bmatrix}, \quad 
        B:=  \begin{bmatrix}
       I & 0 \\
        0 & A
     \end{bmatrix}.
\end{equation}
One easily checks (H1-5).   For (H6), see Section \ref{offdiag}.
If $M= A ( \nabla^k  ) (\nabla^{k} )^*$, then
$$
(BD)^2=\begin{bmatrix}
       L   &  0 \\
         0 & M
     \end{bmatrix},  \quad  \sqrt{ (BD)^2}= \begin{bmatrix}
       \sqrt{L}   &  0 \\
         0 & \sqrt{M} 
     \end{bmatrix} .
     $$
     Since $\sqrt{ (BD)^2} =  \sgn(BD)  BD$, we get for $u\in L^2(\R^n,\C^N)$ under appropriate domain assumptions that,
          $$
  \| \sqrt { L }  u \| =  \left\|\sqrt{(BD)^2} \begin{bmatrix}   u \\ 0 \end{bmatrix}\right\|\approx  \left\|BD \begin{bmatrix}   u \\ 0 \end{bmatrix}\right\| = \|A(\nabla^{k}  u)\| \approx \| \nabla^{k}  u\|.
  $$

 \subsection{Cauchy}
As for the Cauchy integral,  assume $n=m=1$, $D= -i\frac d{dx}$ and $B$ is multiplication by $b(x)=\frac 1 {a(x)}$ where $a\in L^\infty(\R, \C)$ with $\re a \ge \delta>0$.  Then $\sgn (BD)$ is similar to  the Cauchy integral on the Lipschiz curve with parametrization $z(x)$ defined by $z'(x)=a(x)$.

\section{Proof of the main theorem}\label{sec:PMT}

\subsection{Functional calculus for $BD$}

First we need some review on functional calculus. Because of (H2),  $D$ is closed and densely defined and there is an orthogonal  splitting
\begin{equation}\label{eq:splitting}
\mH= \nul(D) \oplus \clos{\ran(D)}.
\end{equation}

Define closed  double sectors in the complex plane by
$$
S_{\omega} := \{z\in\C : {| \pm\arg z|\le\omega}\cup\{0\}\}, 
$$
and define the {\em angle of accretivity} of $B$ to be
$$
   \omega:= \sup_{v\not = 0} |\arg(Bv,v)|  <\pi/2.
$$
\begin{prop}   \label{prop:typeomega}
Under (H2,4,5), we have
\begin{itemize}
\item[{\rm (i)}]
The operator $BD$ is $\omega$-bisectorial, i.e. $\sigma(BD)\subset S_\omega$ and there are resolvent bounds 
$\|(\lambda I - BD)^{-1}\| \lesssim 1/ \dist(\lambda, S_\omega)$ when $\lambda\notin S_\omega$.
\item[{\rm (ii)}]
The operator $BD$ has range $\ran(BD)=B\ran(D)$ and null space $\nul(BD)=\nul(D)$ such that topologically (but in general non-orthogonally) one has
$$
\mH = \clos{\ran(BD)} \oplus \nul(BD).
$$
\item[{\rm (iii)}] 
The restriction of $BD$ to $\clos{\ran(BD)}$ 
is a closed and injective operator with dense range in
$\clos{\ran(BD)}$, with estimates on spectrum and resolvents as in (i).
\end{itemize}
\end{prop}
 These properties of closed operators of the form $BD$ have been known for
some time in the case when $D$ is one-one, see for example \cite{ADMc}. When $D$ is not one-one, first prove (ii), using (\ref{eq:splitting}) and (H5), and then adapt the proof in \cite{ADMc} to prove (iii). Part (i) follows. 
 Note that this proposition only uses the fact that $D$ is self-adjoint and $B$ bounded and strictly accretive on $\ran(D)$.

We set $R_s^B=(I+isBD)^{-1}$ for $s\in \R$. Then  
$$Q_t^B= \frac{1}{2i} (R_{-t^k}^B -R_{t^k}^B)= t^k BD(1+t^{2k}BDBD)^{-1}$$ and also $\frac{1}{2} (R_{-t^k}^B +R_{t^k}^B)= (1+t^{2k} BDBD)^{-1}$.  It follows from the previous result that $R_s^B$, hence $Q_t^B$ 
and $(1+t^{2k} BDBD)^{-1}$, are uniformly bounded operators on $\mH$. 

We  now come to the proof of Theorem \ref{th:main} and  assume all the requirements in (H).

\subsection{Reduction to a Carleson measure}\label{sec:reduction}

Observe that  by item (ii) of Proposition \ref{prop:typeomega}, as $Q_t^B$ vanishes on $ \nul(BD)$ it is enough to prove the quadratic estimate (\ref{eq:sfBD}) for $u\in \clos{\ran(BD)}$, hence for  $u\in {\ran(BD)}$. Setting $\Theta_t=Q_t^B B$, it amounts to showing  $\qe{\Theta_t Dv} \lesssim \|Dv\|^2$ for  all $v\in \dom(D)$. 

Let $P_t$ be a nice approximation of the identity, i.e. the convolution with a real valued function $t^{-n} \varphi(x/t)$ with $\varphi$ smooth and having Fourier transform identically 1 near 0. Let $P_t$ act on $\C^m$-valued function componentwise.

\begin{prop}\label{whatisneeded}  
\begin{equation}  \label{firsttermprop}
     \qe{\Theta _t(\I-P_t) Dv}\lesssim \|Dv\|^2, \quad v \in \dom(D).
\end{equation}
\end{prop}

\begin{proof}\label{proof:thm}  For the purpose of this proof,  by using the splitting
\eqref{eq:splitting}, one can even assume $v\in \clos{\ran(D)}$. Since $P_t$ and $D$ commute and $(I-P_t)v\in \dom(D)$, we have 
$$
\Theta_t (I-P_t) Dv= (\Theta_t D) (I-P_t)v= t^k(BD)^2(I+(t^kBD)^2)^{-1} (I-P_t)v.
$$
Now $(t^kBD)^2(I+(t^kBD)^2)^{-1} =I -(I+(t^kBD)^2)^{-1}$ is uniformly bounded, hence
$$
\|\Theta_t (I-P_t) Dv\| \lesssim \frac 1{ t^k} \| (I-P_t)v\|.
$$
Standard Fourier arguments show that
$$
 \int_0^\infty\|(\I-P_t) v\|^2\,\frac{dt}{t^{2k+1}}\lesssim \|\nabla^k  v\|^2
 $$
 and we conclude the proof of  (\ref{firsttermprop}) using (H3). \qedend
 \end{proof}

 \begin{rem} There are different possible choices of $P_t$'s. For example, following \cite{AKMc} one can take $P_t=(I+t^{2k}D^2)^{-1}$. The organisation of the reduction to a Carleson measure would be somewhat different.
 \end{rem}

 Next, we perform  the principal part approximation.  

We use the following dyadic decomposition of $\R^n$. Let
$\dyadic= \bigcup_{j=-\infty}^\infty\dyadic_{2^j}$ where
$\dyadic_{2^j}:=\{ 2^j(k+(0,1]^n) :k\in\Z^n \}$. For a dyadic cube $Q\in\dyadic_{2^j}$, denote by $l(Q)=2^j$
its \emph{sidelength}, by $|Q|= 2^{nj}$ its {\em volume}.  We set $\dyadic_t=\dyadic_{2^j}$  if $2^{j-1}<t\le
2^j$.  Let the {\em dyadic averaging operator} $S_t:
\mH \rightarrow \mH $ be given by
$$
  S_t u(x) := u_Q:= \barint_{\hspace{-6pt}Q} u(y)\,  dy = \frac{1}{|Q|} \int_Q
  u(y)\,  dy
$$
for every $x \in \R^n$ and $t>0$, where  $Q $ is
the unique dyadic cube  in $ \dyadic_t$ that contains $x$. We remark that $S^2_t=S_t$.

\begin{defn}  \label{defn:princpart}
By the {\em principal part} of $(\Theta_t)_{t>0}$
we mean the multiplication operators $\gamma_t$ defined by
$$
   \gamma_t(x)w:= (\Theta_t w)(x)
$$
for every $w\in \C^m$. We view $w$ on the right-hand side
of the above equation as the constant function valued in $\C^m$ defined on $\R^n$ 
by $w(x):=w$. 
We identify $\gamma_t(x)$ with the (possibly unbounded) multiplication
operator $\gamma_t: f(x)\mapsto \gamma_t(x)f(x)$.
\end{defn}
\begin{lem}\label{lem:gammat}  
The operator $\Theta_t$ extends to a bounded operator from 
$L^\infty$ into $ L^2_{\text{loc}}$.
In particular we have well defined functions 
$\gamma_t\in L^2_{\text{loc}}(\R^n; \mL(\C^m,\C^m))$ with bounds
$$
   \barint_{\hspace{-6pt}Q} |\gamma_t(y)|^2 \, d y \lesssim  
  1
$$
for all $Q\in\dyadic_t$.  
Moreover, $\|\gamma_t S_t\|\lesssim 1$ uniformly for all $t>0$.
\end{lem}

\begin{proof} Fix a cube $Q \in \dyadic_t$ and $f \in L^\infty(\R^n,\C^m)$ with $\|f\|_\infty=1$. Then
write $f= f_0+ f_1+f_2+\ldots$ where $f_0=f$ on $2Q$ and $0$ elsewhere and if $j\ge 1$, $f_j=f$ on $2^{j+1}Q \setminus 2^{j}Q$ and $0$ elsewhere. Then apply $\Theta_t$  and use (H6) for each term $\Theta_t f_j$ with $N$ large enough  and sum  to obtain
$$
   \barint_{\hspace{-6pt}Q} |(\Theta_tf)(y)|^2 \, d y \le  
  C.
$$
If we do this for  the constant functions with values describing an orthonormal basis of $\C^m$ and sum, we obtain an upper bound for the desired average of $\gamma_t$.   
Next,  for a function $f\in \mH$, 
$$
\|\gamma_t S_tf\|^2 = \sum_{Q\in \dyadic_t} \int_Q \left|\gamma_t (y)\bigg( \barint_{\hspace{-6pt}Q} f\bigg)\right|^2 \, dy \lesssim  \sum_{Q\in \dyadic_t}  |Q|\left|\,\barint_{\hspace{-6pt}Q} f\right |^2 \le \|f\|^2.
$$
\end{proof}
We have the following principal part approximation of $\Theta_t$ by $\gamma_tS_t$.
\begin{lem}  \label{lem:ppa} We have
 \begin{equation}\label{eq:ppa1}
  \qe{\Theta_tP_tf-\gamma_t S_t f}
  \lesssim \|f\|^2,  \quad f\in \mH. 
\end{equation}
Combining this with Proposition \ref{whatisneeded}, we obtain the principal part approximation 
\begin{equation}\label{eq:ppa}
 \qe{\Theta_t Dv-\gamma_t S_t Dv}
  \lesssim \|Dv\|^2, \quad v\in \dom(D).
\end{equation}
\end{lem}

\begin{proof}  Write 
$$\Theta_tP_t -\gamma_t S_t = (\Theta_tP_t -\gamma_tS_tP_t )+ (\gamma_tS_t(P_t-S_t)) + (\gamma_tS_t^2 - \gamma_tS_t).
$$
Because $S_t^2=S_t$, the last term vanishes.  Next, as $\gamma_tS_t$ is uniformly bounded as an operator on $\mH$, we have 
$$
\qe{\gamma_tS_t(P_t-S_t)f} \lesssim \qe{(P_t-S_t)f} \lesssim \|f\|^2.
$$
The last inequality is done componentwise and is classical  (See, e.g. \cite{AT}, p. 172).

We pass to the first term.
We remark that for $t>0$ fixed and $x\in \R^n$, then 
$$(\Theta_tP_t -\gamma_tS_tP_t )f(x)= \Theta_t \bigg(g- \barint_{\hspace{-6pt}Q} g \bigg) (x)
$$
where $g=P_tf$ and $Q$ is the only dyadic cube in $\dyadic_t$ containing $x$. 
Define $C_0(Q)=2Q$ and $C_j(Q)=2^{j+1}Q\setminus 2^jQ$ if $j\in \N^*$. Then
\begin{align*}
\|(\Theta_tP_t -\gamma_tS_tP_t )f \|^2 &=  \sum_{Q\in \dyadic_t} \int_Q \bigg| \Theta_t\bigg(g- \barint_{\hspace{-6pt}Q} g \bigg)\bigg|^2 
\\
& \le \sum_{Q\in \dyadic_t} \left(\sum_{j\ge 0} \bigg(\int_Q \bigg| \Theta_t \bigg({\bf 1}_{C_j(Q)} \bigg(g- \barint_{\hspace{-6pt}Q} g \bigg)\bigg)\bigg|^2\ \bigg)^{1/2}\right)^{2}
\\
&
\lesssim  \sum_{Q\in \dyadic_t} \left(\sum_{j\ge 0} 2^{-jN} \bigg(\int_{C_j(Q)}\bigg|g- \barint_{\hspace{-6pt}Q} g \bigg|^2\ \bigg)^{1/2}\right)^{2}
\\
&
\lesssim  \sum_{Q\in \dyadic_t} \sum_{j\ge 0}  2^{-jN} \int_{C_j(Q)}  \bigg|g- \barint_{\hspace{-6pt}Q} g \bigg|^2
\\
&
\lesssim  \sum_{Q\in \dyadic_t} \sum_{j\ge 0}  2^{-jN} 2^{2j} \ell(Q)^2 \int_{2^{j+1}Q} |\nabla  g |^2
\\
&
\lesssim   t^2 \sum_{j\ge 0}  2^{-jN} 2^{2j} 2^{jn} \int_{\R^n} |\nabla  g |^2.
\\
&
\lesssim   t^2 \|\nabla  g \|^2.
\end{align*} 
We successively used the Minkowski inequality on the second line, (H6) on the third one,  Cauchy--Schwarz on the fourth,
Poincar\'e inequality on the fifth, 
the covering inequality $ \sum_{Q\in \dyadic_t}  {\bf 1}_{2^{j+1}Q} \lesssim 2^{jn}$ and $\ell(Q)\sim t$  on the sixth
and the choice $N> n+2$ in the last. 

Hence
$$\qe{\Theta_tP_tf-\gamma_t S_t P_t f}
  \lesssim \qe{t \nabla  P_t f} \lesssim \|f\|^2
  $$
  using the standard Littlewood-Paley inequality on each component of $f$. 
  \end{proof} 
  
  Before we state the conclusion of this reduction, there is an essential observation. Identifying constant functions with their values, observe  that $Dv$ takes values in the  vector space  $$\mD=\{DL: L\colon \R^n \to \C^m, L\  \text{a  polynomial of degree }k\} \subset \C^m$$  and so does $S_t(Dv)$. Therefore, one  considers the restriction of $\gamma_t(x)$ to $\mD$. Henceforth, we consider  $\gamma_t(x)$ as an element of $\mL(\mD, \C^m)$ and its norm $|\gamma_t(x)|$ is measured in this space.

Recall that  $\R^{n+1}_+\rightarrow \mL(\mD,\C^m), (x,t) \mapsto \gamma_t(x)$, is 
a {\em dyadic Carleson function} if there exists $C<\infty$
such that
$$
  \iint_{R(Q)} |\gamma_t(x)|^2 \frac{dxdt}t \le C^2 |Q|
$$
for all dyadic cubes $Q\subset\R^n$.
Here $R(Q):= Q\times (0,l(Q)]$ is the Carleson box over $Q$. We define the dyadic Carleson norm $\|\gamma_t\|_C$ to be the smallest
constant $C$.  The form of Carleson's lemma that we need and applied componentwise is as follows  (see \cite{AT}, p.168 and references therein). 
\begin{prop}  \label{lem:Carleson}
$$
  \qe{\gamma_t S_tDv} \lesssim \|\gamma_t\|_C^2 \| Dv\|^2, \quad v\in \dom(D).
$$
\end{prop}

Therefore, we have obtained

\begin{prop} If  the restriction of $\gamma_t(x)$ to $\mD$ is a dyadic Carleson function then the conclusion of Theorem \ref{th:main} holds. 
\end{prop}

\begin{rem}   At this point, it is nowadays understood that the Carleson measure estimate can be achieved by what is called a $T(b)$ argument, which consists in finding suitable test functions adapted to the operator $\Theta_{t}=Q_{t}B$. However, a dichotomy appears on remarking that  we can prove \eqref{eq:ppa} for functions of the form $f=Dv\in \clos{\ran(D)}$, but not for functions $f\in \nul(D)$. 
This comes from the use of (H3) in Proposition~\ref{whatisneeded}.  The simple situation is when $D$ is one-one (or, equivalently, $D$ has dense range by \eqref{eq:splitting}): the test functions are simply the columns of $B^{-1}$. 
When $D$ fails to be one-one, this choice does not work as we have to select test functions in the range of $D$.  For the  Kato problem, one has one-oneness of the $D$ involved only in one dimension.  It is for this reason that the Kato problem was more difficult in dimensions $n\geq2$ than in one dimension.  For a fair comparison, we provide the concluding argument of the proof of Theorem \ref{th:main} in  both cases.
\end{rem}

\subsection{The T(b) argument when $D$ is injective with dense range}\label{sec:injective}    

Fix a dyadic cube $Q$, and let $\eta_Q$
be a smooth real valued cutoff such that $\eta_Q|_{2Q}=1$, $\supp(\eta_Q)\subset 3Q$
and $\|\nabla^j \eta_Q\|_\infty\lesssim l^{-j}$ for $j=1,2,\ldots, k$ with $l=l(Q)$.
Denote by $B^{-1}_j$ the $j$'th column vector in the matrix $B^{-1}$, and estimate
\begin{multline*}
  \iint_{R(Q)} |\gamma_t(x)|^2 \frac{dxdt}t
\lesssim   \iint_{R(Q)} |\gamma_t(x)S_t(\eta_Q B^{-1})|^2 \frac{dxdt}t \\
\lesssim   \sum_{j=1}^m\iint_{R(Q)} |\gamma_t(x)S_t(\eta_Q B^{-1}_j)|^2 \frac{dxdt}t \\
\lesssim   \sum_{j=1}^m\iint_{R(Q)} |\Theta_t(\eta_Q B^{-1}_j)|^2 \frac{dxdt}t +
 \sum_{j=1}^m\int_0^\infty \|(\Theta_t-\gamma_t(x)S_t)(\eta_Q B^{-1}_j)\|^2 \frac{dt}t \\
\lesssim \sum_{j=1}^m\int_0^{l(Q)} \| (1+t^{2k}(BD)^2)^{-1} t^k BD(\eta_Q e_j)\|^2 \frac{dt}t +  \sum_{j=1}^m \|\eta_Q B^{-1}_j\|^2\\
\lesssim \sum_{j=1}^m\int_0^{l(Q)} \|D(\eta_Q e_j)\|^2 t^{2k-1} dt + |Q| \lesssim |Q|.
\end{multline*}
For the first row, we use the fact that $B^{-1}$ is strictly accretive on $\mH$, and hence is pointwise uniformly strictly accretive. Here
$S_t$ acts componentwise on the 
matrix.
In the first term of row four we write $B\eta_Q B^{-1}_j= \eta_Q BB^{-1}_j =\eta_Q e_j$, where $e_j$ is the $j$'th standard basis vector  in $\C^m$. To obtain the second term, we 
apply the principal part approximation  (\ref{eq:ppa}), using the assumption that
the range of $D$ is dense in $\mH$ (hence $Dv$  there can be replaced by {\bf any}  function in $\mH$).
In row five we use the uniform boundedness of the operators $(1+t^{2k}(BD)^2)^{-1}$
and that $D(\eta_Q e_j)$ is supported on $3Q$ and is bounded by $l^{-k}$.

\subsection{The T(b) argument in the general case}    
\label{sec:tb}
 We now consider the general case where $D$ is not an injective operator with dense range in $\mH$, so that we need
to construct test functions which belong to the range of $D$.
Fix $Q$ a dyadic cube and $w\in \mD$ with $|w|=1$. 
 Let  $L$ be a  polynomial of degree $k$  such that 
$w= D L$ and $\sup_{3Q}|\partial^\alpha L(x)|\lesssim l^{k-|\alpha|}$, $0\leq |\alpha|\leq k-1$  and
define $w_Q := D(\eta_Q L)$,  where $\eta_Q$ is the cutoff above. 
It follows that
$$
  w_Q\in\ran(D),\quad
  w_Q|_{2Q}= w,\quad
  \supp w_Q\subset 3Q\quad\text{and}\quad 
  \|w_Q\|_\infty\le C.
$$
Next we define the test function $b^w_{Q,\epsilon}$ for $\epsilon \in (0,1)$  by
$$  b^w_{Q,\epsilon}: = Dv^w_{Q,\epsilon}, \quad v^w_{Q,\epsilon}:= (I+i(\epsilon l)^k BD)^{-1}(\eta_QL).  
$$
\begin{lem}\label{lem:testfunctions} 
 There exists $C>0$ such that  for each $w\in \mD$ with $|w|=1$, each dyadic cube $Q\subset\R^n$ and each $\epsilon\in (0,1)$,
 \begin{equation}\label{eq:1}
\int_Q|v^w _{Q,\epsilon} - L|^2  \le C (\epsilon l)^{2k}|Q|,
\end{equation}
\begin{equation}\label{eq:2}
\int_Q|b^w _{Q,\epsilon} - w|^2  \le C |Q|,
\end{equation}
\begin{equation}\label{eq:3}
 \left| \barint_{\hspace{-6pt}Q} b^w_{Q,\epsilon} - w \right| \le
   C  \sqrt\epsilon,
\end{equation}
\begin{equation}\label{eq:4}
\iint_{R(Q)} |\gamma_t(x) S_tb^w _{Q,\epsilon}(x)|^2 \, \frac{dxdt}{t} \le C\epsilon^{-2k} |Q|.
\end{equation}
\end{lem}
\begin{proof} Using $(I+is BD)^{-1}-I = -isBD(I+is BD)^{-1}$ and  $\eta_QL\in \dom(BD)$, we have 
$$ v^w _{Q,\epsilon} - \eta_QL= -i(\epsilon l)^k(I+i(\epsilon l)^k BD)^{-1}(Bw_Q).$$
The properties of $w_Q$ and $\eta_Q$ and the boundedness of  $(I+is BD)^{-1}B$ imply  \eqref{eq:1}. Applying $D$ we get, 
$$
b^w _{Q,\epsilon} - w_Q= -i(\epsilon l)^k D(I+i(\epsilon l)^k BD)^{-1}(Bw_Q).
$$
The properties of $w_Q$ and the boundedness of $s D(I+is BD)^{-1}B$ imply \eqref{eq:2}.

Next, let $\varphi\colon \R^n\to [0,1]$ be a smooth function which is 1 on $(1-t)Q$, 0 on $Q^c$
with $\|\nabla^k \varphi\|_\infty \le C(tl)^{-k}$ with $t\in (0,1)$ to be chosen. We can write
$$
\barint_{\hspace{-6pt}Q} b^w_{Q,\epsilon} - w
= \barint_{\hspace{-6pt}Q} \varphi D(v^w_{Q,\epsilon} - L) + \barint_{\hspace{-6pt}Q} (1-\varphi)(b^w_{Q,\epsilon} - w)
=I+II.
$$
Using \eqref{eq:2} and the properties of $\varphi$ together with Cauchy-Schwarz inequality, we obtain
$$
| II | \le C\sqrt t.
$$
For $I$, we can write using the properties  $\varphi$ and integration by parts,
$$
\int_Q \varphi D(v^w_{Q,\epsilon} - L)= \int_{\R^n} \varphi D(v^w_{Q,\epsilon} - \eta_QL) = \int_{\R^n} (\widetilde{D}\varphi)(v^w_{Q,\epsilon} - \eta_QL) = \int_{Q} (\widetilde{D}\varphi)(v^w_{Q,\epsilon} - L)
$$ where $\widetilde{D}\varphi$ is some $\mL(\C^m,\C^m)$-valued function bounded by $C\|\nabla^k \varphi\|_\infty$ and supported in $Q\setminus (1-t)Q$, so that  we obtain
$$
| I  | \le  C\epsilon^k/ t^{k-1/2}.
$$
Hence, choosing $t=\epsilon$, we have shown \eqref{eq:3}.

Eventually, to prove \eqref{eq:4}, we can use the principal part approximation in Lemma \ref{lem:ppa} (backwards) because $b^w _{Q,\epsilon}=Dv^w _{Q,\epsilon}$ and $\|b^w _{Q,\epsilon}\| \lesssim 1$
and it suffices to establish
\begin{equation}\label{eq:5}
\iint_{R(Q)} |\Theta_tb^w _{Q,\epsilon}(x)|^2 \, \frac{dxdt}{t} \le C\epsilon^{-2k} |Q|.
\end{equation}
Now, 
\begin{align*}
\Theta_tb^w _{Q,\epsilon}&= t^kBD(I + t^{2k}BDBD)^{-1}BD(I+i(\epsilon l)^k BD)^{-1}(\eta_QL)
\\
&= t^kBD(I + t^{2k}BDBD)^{-1}(I+i(\epsilon l)^k BD)^{-1}(Bw_Q).
\\
&
=(t/\epsilon l)^k (I + t^{2k}BDBD)^{-1}(\epsilon l)^k BD(I+i(\epsilon l)^k BD)^{-1}(Bw_Q)
\end{align*}
Since $(I + t^{2k}BDBD)^{-1}$ and $(\epsilon l)^k BD(I+i(\epsilon l)^k BD)^{-1}$ are bounded uniformly with respect to $t$ and $\epsilon l$, we have
$$
\|\Theta_tb^w _{Q,\epsilon}\| \le C(t/\epsilon l)^k.
$$
Integrating over $t\in (0, l]$ we obtain \eqref{eq:5}.
\end{proof}

We now perform a sectorial decomposition and then a stopping-time argument to estimate the dyadic Carleson norm on $\gamma_t(x)$. Cover $\mL(\mD,\C^m)$ by a finite number of sectors 
$\mC_{\gamma, \nu} =\{\kappa \in  \mL(\mD,\C^m) \, ; \, |\kappa - |\kappa| \gamma| \le \nu |\kappa|\}$, with $\gamma\in \mL(\mD,\C^m)$, $|\gamma|=1$, and $\nu\in (0,1)$. The number $\nu$ is to be chosen later.
Fix such a sector. It is enough to estimate the Carleson norm of
$$
\tilde \gamma_t(x)= {\bf 1}_{\gamma_t(x) \in \mC_{\gamma, \nu}} \gamma_t(x).
$$
Pick $w\in \mD, w^*\in \C^m$ such that $(\gamma w, w^*)=1$ and $|w|=|w^*|=1$. For any $\kappa \in 
\mC_{\gamma, \nu}$, we have
$$
\re \big[|\kappa| (\gamma w, w^*) - (\kappa w, w^*) \big] \le \nu |\kappa|
$$
thus
$$
(1-\nu) |\kappa| \le \re (\kappa w, w^*).
$$
Fix a cube $Q$. Applying this to $\tilde \gamma_t(x)$ with $(x,t)\in R(Q)$, we obtain
\begin{align*}
(1-\nu)|\tilde \gamma_t(x)| 
&\le   \re (\tilde \gamma_t(x) w, w^*)
\\
&\le \re (\tilde \gamma_t(x) S_tb_{Q,\epsilon}^w(x), w^*) + 
|\tilde \gamma_t(x)| \re ( \gamma (w-S_tb_{Q,\epsilon}^w(x)), w^*)\\
& \qquad \qquad +
\re ((\tilde \gamma_t(x)-|\tilde \gamma_t(x)|  \gamma) (w-S_tb_{Q,\epsilon}^w(x)), w^*)
\\
&\le |\gamma_t(x) S_tb_{Q,\epsilon}^w(x)| + |\tilde \gamma_t(x)| \re ( \gamma (w-S_tb_{Q,\epsilon}^w(x)), w^*) 
\\
&\qquad\qquad + \nu |\tilde \gamma_t(x)| |w-S_tb_{Q,\epsilon}^w(x)|.
\end{align*}

Thus one needs smallness on $\re ( \gamma (w-S_tb_{Q,\epsilon}^w(x)), w^*) $ and a control
on the size of $|w-S_tb_{Q,\epsilon}^w(x))|$ on a large portion of $R(Q)$.

\begin{lem}\label{lem:stopping}  There exists $\epsilon_0 \in (0,1)$ such that for all $\epsilon\in (0,\epsilon_0)$,  any dyadic cube $Q$ contains disjoint dyadic subcubes $Q_i$ with 
\begin{equation} \label{eq:stopping1}
\sum_{i} |Q_i| \le \left(1-\epsilon\right) |Q|,
\end{equation}
\begin{equation}\label{eq:stopping2}
\re ( \gamma (w-S_tb_{Q,\epsilon}^w(x)), w^*)  \le 10C  \sqrt \epsilon, \quad (x,t) \in R(Q) \setminus  \cup R(Q_i),
\end{equation}
\begin{equation}\label{eq:stopping3}
|w-S_tb_{Q,\epsilon}^w(x)|  \le \sqrt{C/\epsilon},  \quad (x,t) \in R(Q) \setminus  \cup R(Q_i).
\end{equation}
Here, $C$ is the constant appearing in Lemma  \ref{lem:testfunctions}.
\end{lem}
Assuming this, then we obtain
$$
\left(1-\nu - 10C \sqrt\epsilon - \nu \sqrt{C/\epsilon}\right) |\tilde \gamma_t(x)| \le |\gamma_t(x) S_tb_{Q,\epsilon}^w(x)|,  \quad (x,t) \in R(Q) \setminus  \cup R(Q_i).
$$
Choosing  $\epsilon$ and  then $\nu$ small enough (depending only on $C$, hence on (H)), we have shown for all $Q$ with the corresponding $Q_i$
\begin{equation} \label{eq:ref1}
 |\tilde \gamma_t(x)| \le 2|\gamma_t(x) S_tb_{Q,\epsilon}^w(x)|,  \quad (x,t) \in R(Q) \setminus  \cup R(Q_i).
 \end{equation}

We finish with a classical  observation: fix $\delta>0$ and $$A_Q= \sup \frac{1}{|Q'|} \iint_{(x,t) \in R(Q'), t>\delta} |\tilde \gamma_t(x)|^2 \, \frac{dxdt}{t} <\infty$$
where the supremum is taken over all dyadic subcubes of $Q$. 
Then, if $Q'$ is such a cube and  $Q_i'$ are the subcubes of $Q'$ given by Lemma \ref{lem:stopping}
\begin{equation}\label{eq:ref2}
\begin{split}
&\iint_{(x,t) \in R(Q'), t>\delta} |\tilde \gamma_t(x)|^2 \, \frac{dxdt}{t} \\
&
\le 4 \iint_{(x,t) \in R(Q'), t>\delta} |\gamma_t(x) S_tb_{Q',\epsilon}^w(x)|^2 \, \frac{dxdt}{t} 
+ \sum_i \iint_{(x,t) \in R(Q_i'), t>\delta} |\tilde \gamma_t(x)|^2 \, \frac{dxdt}{t}
\\
&
\le 4 C\epsilon^{-2k} |Q'| + A_Q \sum_i |Q_i'|
\\
&
\le 4 C\epsilon^{-2k} |Q'| + A_Q (1-\epsilon) |Q'|.
\end{split}
\end{equation}
Hence, dividing by $|Q'|$ and taking the supremum over $Q'$ we obtain $A_Q\le 4C\epsilon^{-2k-1}$, 
and in particular
$$
\frac{1}{|Q|} \iint_{(x,t) \in R(Q), t>\delta} |\tilde \gamma_t(x)|^2 \, \frac{dxdt}{t} \le 4C\epsilon^{-2k-1}.
$$This is independent of $\delta>0$, hence we obtain the desired estimate by letting $\delta$ tend to 0. 

It remains to prove Lemma \ref{lem:stopping}.

\begin{proof} We fix a dyadic cube $Q$. We assume $\epsilon$ small. Observe that 
$$
\re (\gamma \left(w-\barint_{\hspace{-6pt}Q} b^w_{Q,\epsilon}\right), w^*) \le C \epsilon^{1/2}
$$
and 
$$
\barint_{\hspace{-6pt}Q}|w-b^w _{Q,\epsilon}|^2  \le   C.$$
We subdivide dyadically $Q$ and consider for the  subcubes $Q'$ both conditions
 \begin{equation}
\label{eq:type1}
\re (\gamma \left(w-\barint_{\hspace{-6pt}Q'} b^w_{Q,\epsilon}\right), w^*) > 10 C \epsilon^{1/2},
\end{equation}
\begin{equation}
\label{eq:type2}
\barint_{\hspace{-6pt}Q'}|w-b^w _{Q,\epsilon}|^2  >  C\epsilon^{-1}.
\end{equation}
If one or the other holds, we stop and put $Q'$ in the sought collection of stopping cubes  $(Q_i)$. If none of the conditions hold, we subdivide $Q'$ and iterate the test on subcubes. 

We note that $ (x,t) \in R(Q) \setminus  \cup R(Q_i)$ exactly means that  $w-S_tb_{Q,\epsilon}^w(x)= w-\barint_{\hspace{-3pt}Q'} b^w_{Q,\epsilon}$ for a non-stopping cube $Q'$. Thus \eqref{eq:stopping2} and \eqref{eq:stopping3} hold immediately.

It remains to show \eqref{eq:stopping1}. Declare $Q_i$  of type 1 if \eqref{eq:type1} holds and
of type 2 if \eqref{eq:type2} holds. We let $\Sigma_j= \sum |Q_i|$ where the sum is retricted to cubes of type $j$. We might count twice cubes of both types  but that is not a problem. For cubes of type 2, we have
$$
\Sigma_2 \le \frac {\epsilon} C \sum \int_{Q_i}|w-b^w _{Q,\epsilon}|^2 \le  \frac \epsilon C  \int_{Q}|w-b^w _{Q,\epsilon}|^2 \le  {\epsilon}  |Q|.
$$
For cubes of type 1, we have
\begin{align*}
10C \epsilon^{1/2} \ \Sigma_1 &\le  \sum \re (\gamma \bigg(\int_{Q_i} w- b^w_{Q,\epsilon} \bigg), w^*)
\\
&=  \re (\gamma \bigg(\int_{Q} w- b^w_{Q,\epsilon}\bigg), w^*) - \re (\gamma\bigg(\int_{Q\setminus \cup Q_i} w- b^w_{Q,\epsilon}\bigg), w^*).
\end{align*}
Using \eqref{eq:3} and the Cauchy-Schwarz inequality for the last term, we obtain
$$
10C\epsilon^{1/2} X \le C\epsilon^{1/2}  + \sqrt C (1-X)^{1/2}
$$
where $X = \Sigma_1/ |Q| \in [0,1]$. The positive root of the corresponding equation is on the order of $1-81 C \epsilon$ for $\epsilon $ small enough. Hence, 
$$
X\le 1- C\epsilon
$$
for $\epsilon $ small enough.
Thus, the total contribution of cubes of both types does not exceed
$
 (1-C\epsilon +\epsilon) |Q|,
$
which gives \eqref{eq:stopping1}  (assuming $C\ge 2$ which we may).
 \end{proof}

 \section{Historical comments} The (almost) self-contained proof of Theorem \ref{th:main}  follows very closely the strategy of \cite{AHLMcT} which, of course, builds upon the ideas of many authors, and it also incorporates ideas from the various extensions of this argument found later on.  Locating the origin of this and that  can be subtle for the reader, so we devote this section to historical comments,  giving appropriate credit for the crucial steps based on our understanding. We do not mention less recent progress and refer to \cite{AT, AHLMcT} for this. 
 
A strategy to solve the Kato conjecture in all dimensions was introduced and developed in \cite{AT} under kernel bound assumptions. It first involves the reduction to Carleson measures in such a context - which was named ``principal part approximation'' in \cite{AKMc} -  as  described in 
 Section \ref{sec:reduction}, exploiting earlier ideas of Coifman-Meyer \cite{CM} further elaborated in works of Christ-Journ\'e and Semmes \cite{CJ,S}. The present formulation of the principal part approximation is closer to the one in \cite{AKMc}. We have chosen this formulation for the simplicity of its proof (assuming minimal knowledge of Littlewood-Paley-Stein theory). The strategy of \cite{AT} required  the existence of a set of appropriate test functions  in order to prove the Carleson bounds via the 
  ``T(b) theorem for square roots''. In \cite{AT}, Chapter 3 this existence was made an assumption called there ``the class $(S)$ assumption''. The construction of such a set was achieved for the first time in \cite{HMc} to  solve the two  dimensional Kato problem. Our choice is close to this one,  rather than the one used later in \cite{AKMc}. 
 However, we need to exploit the observation made in \cite{AAH} that one can reduce the action 
 of $\gamma_{t}(x)$ to a subspace, while this is not necessary in \cite{HMc} or in \cite{AHLMcT}.  
 
 The importance of the inequality \eqref{eq:ref1}, or at least an integral version of it, was pointed out in \cite{AT}, Chapter 3, and found its roots in  \cite{S}.
 The kind of stopping-time argument providing  an inequality like \eqref{eq:ref1} leading to \eqref{eq:ref2} is developed for the first time in this context in \cite{HMc}. It is mentioned in \cite{AHLMcT} that, in retrospect, this stopping-time argument is akin to an argument of Christ \cite{Ch} devised for proving a local $T(b)$ theorem for singular integrals. The conical decomposition done in the space of constants ($\C^n$) to estimate the Carleson measure associated to $\gamma_{t}(x)$ was the main new ingredient of \cite{HLMc}. This provided a means to build a different set of test functions to solve the Kato conjecture in all dimensions  under kernel bound assumptions. The removal of such kernel bounds was achieved in \cite{AHLMcT}, thus proving the Kato conjecture for second order operators in full generality. The idea which we use of doing  the conical decomposition, not in the space of  constants, but within the 
 linear space of matrices to which $\gamma_{t}(x)$ belongs, is an important observation, made in \cite{AHMcT} for proving the Kato conjecture for higher order operators and systems. 
 
 Note that our argument is developed on $\R^n$, while the one in \cite{AHLMcT} was pushed in \cite{AKMc1} to Lipschitz domains for mixed boundary value problems. The case of Dirichlet and Neumann boundary conditions had been previously done in \cite{AT1} by a direct reduction to the $\R^n$ case.
It would be of interest to adapt Theorem \ref{th:main} to domains and to obtain  new proofs and generalisations of the results just mentioned.

\section{Validity of off-diagonal estimates}\label{offdiag}

\begin{prop}  \label{pseudoloc} 
If $k=1$ then (H6) holds for all $B$ and $D$ with (H1,2,4,5) when also $B$ denotes multiplication by a matrix-valued function $B\in L^\infty(\R^n, \mathcal L(\C^m))$.  In fact, one even has exponential decay.
\end{prop}

The proof is inspired by the one in \cite{AHLMcT}.

\begin{proof} It is enough to consider $R^B_t=(I+i t BD)^{-1}$ for $t\ne 0$, as 
$Q_t^B=\frac{-1}{2i}(R_t^B-R_{-t}^B)$.  Let  $d=\dist(E,F)$. We have already proved uniform bounds. So it is enough to prove (\ref{odn}) under the
assumption  that $|t| \le \alpha d$ for some constant $\alpha>0$ to be chosen. 
Assume  $u \in \mH$ with $\supp u\subset F$.

   Write $$\widetilde E:=\{x\in\R^n :
\dist(x,E)<\tfrac12\dist(x,F) \}$$ and let
$\varphi:\R^n\longrightarrow[0,1]$ be a Lipschitz function such
that $\supp\varphi\subset\widetilde E$, $\varphi|_E=1$ and
$$\|\nabla\varphi\|_\infty\leq 4/d.$$  Let $\eta= e^{\alpha d \varphi /t} -1$ and observe that 
$\eta= 0$ on $F$ and $\eta= e^{\alpha d /t} - 1 \ge \frac 1 2 e^{\alpha d /t}$ on $E$, thus
$$
\frac 1 2 e^{\alpha d /t}\|R_t^B
u\|_{L^2(E)} \le   \| \eta R_t^Bu\| \le \|   [\eta \I,R_t^B]u\|
$$
using that $\eta u = 0$. 
Next,
$$
       [\eta \I,R_t^B]= itR_t^B[BD, \eta \I]R_t^B= itR_t^BB [D, \eta \I] R_t ^B
$$
and $[D, \eta \I]$ is multiplication by a function $me^{\alpha d \varphi/t}$ where $m$ is supported on $\widetilde E$ with $L^\infty$ norm not exceeding $C\alpha d \|\nabla \varphi \|_\infty/t \le 4C\alpha/t$.
Thus, using the boundedness of $R_t^BB$, 
  \begin{equation*} \begin{split} 
\|\eta R_t^Bu\| 
     & \lesssim t \|[D, \eta \I]R_t^Bu\| 
     \\
     & 
     \lesssim  4C\alpha \left\|e^{\alpha d \varphi /t}R_t^Bu\right\|
     \\
     &
     \lesssim  4C\alpha (\|\eta R_t^Bu\| + \|R_t^Bu\|).
     \end{split} \end{equation*}

Hence, choosing $\alpha$ small enough (independent of $t, u$), gives
$\|\eta R_t^B u\| \lesssim \|R_t^Bu\| \lesssim \|u\|$ and  this proves the proposition.
 \qedend
\end{proof}

\begin{prop} Let $k\ge 2$ and $B$, $D$ be as in Section \ref{sec:K}. Then (H6) holds with exponential decay.
\end{prop}

\begin{proof} Observe that 
$$
 Q_t^B \begin{bmatrix}
        u_1 \\
       u_2
     \end{bmatrix} =  \begin{bmatrix}
      (I+t^{2k}L)^{-1}  (t^k (\nabla^{k})^*  u_2) \\
       t^kA \nabla^k (I+t^{2k}L)^{-1}  u_1
     \end{bmatrix}. 
     $$

 The off-diagonal bounds \eqref{odn} for $(I+t^{2k}L)^{-1} $ and $t^k \nabla^k (I+t^{2k}L)^{-1}$ have been known for some time:  see \cite{Da}  where it is done for the semi-group $e^{-t^{2k}L}$ instead of the resolvent.  However, there is an argument using  the spirit of the proof of Proposition~\ref{pseudoloc} working directly with $Q_t^B$ instead of $R_t$.  From there the off-diagonal bounds  for  $(I+t^{2k}L)^{-1}  t^k (\nabla^{k})^* A$ follow from a duality argument changing $A^*$ to $A$.  We leave details to the reader.  
 \end{proof}
\section{ Some functional consequences of the quadratic estimate }

\subsection{Proof of  Proposition \ref{cor:main}}\label{sec:ode}

We refer to \cite{ADMc} for details on functional calculus for the class of operators under consideration here. Let us just say that there is a way of defining $\sgn(BD)$ using the 
following formula
\begin{equation}\label{eq:sgnBD}
\sgn(BD)f= c\int_0^\infty (t^kBD)^3(1+t^{2k}BDBD)^{-3}f\,  \frac {dt}{t}= c\int_0^\infty (Q_t^B)^3 f \,  \frac {dt}{t}
\end{equation}
with $c^{-1}=\int_0^\infty u^{3k-1}(1 +u^{2k})^{-3} \, du.$ This comes from the fact that
the function $z\mapsto c \int_0^\infty (t^kz)^3(1+t^{2k}z^2)^{-3}\, \frac {dt}{t}$ is holomorphic 
on $\C \setminus i\R$ where it coincides with $\sgn(z)$, defined to be 1 on the right half-plane and -1 on the left half-plane, in other words, the holomorphic extension of the $\sgn$ function on the real line to 
$\C \setminus i\R$. 

By item (ii), Proposition \ref{prop:typeomega}, it is enough to define and prove boundedness of $\sgn(BD)$ on $\nul(BD)$ and $\clos{\ran(BD)}$ separately. For $f\in \nul(BD)$ then $Q_t^B f=0$ for each $t$, thus $\sgn(BD)=0$ on $\nul(BD)$. 

 It is easy to see that the integral \eqref{eq:sgnBD} converges in norm in $\mH$ for $f\in \dom(BD)\cap \ran(BD)$, because then 
$$ \| (t^kBD)^3(1+t^{2k}BDBD)^{-3}f\| \lesssim \min (t^k, t^{-k}).
$$ Since $ \dom(BD)\cap \ran(BD)$ is dense in $\clos{\ran(BD)}$, this defines $\sgn(BD)$ on the latter provided one shows $\|\sgn(BD)f \| \le c \|f\|$ for $f\in \dom(BD)\cap \ran(BD)$.

Let $f\in\dom(BD),g \in \dom(DB^*)$. Then 
$$
|((Q_t^B)^3f, g) | = |(Q_t^B (Q_t^Bf), (Q_t^B)^*g)| \lesssim \|Q_t^Bf\|\|(Q_t^B)^*g\|
$$
and applying the Cauchy-Schwarz inequality
$$
|(\sgn(BD) f, g)| \lesssim \left( \qe{Q_t^Bf}\right)^{1/2} \left( \qe{(Q_t^B)^*g}\right)^{1/2}.
$$
The first factor is directly controlled by $c\|f\|$ by assumption. For the second factor,  write
$$(Q_t^B)^*g= (I + t^{2k}DB^*DB^*)^{-1} t^kDB^*g= t^kDB^*(I + t^{2k}DB^*DB^*)^{-1} g.$$
We shall show in a moment that quadratic estimates for operators $DB^*$ are a consequence of the assumed quadratic estimates for $B^*D$. We conclude that  $\sgn(BD)$ is bounded as desired.
       
   We remark that $ \sgn(BD)\sgn(BD)= I$ on $\clos{\ran(BD)}$ from the properties of functional calculus. This gives the invertibility of $\sgn(BD)$ on $\clos{\ran(BD)}$, and the proposition is proved. 
   
\subsection{Operators of type $DB$}

\begin{prop}   \label{prop:typeomega2}
Under (H2,4,5), we have
\begin{itemize}
\item[{\rm (i)}]
The operator $DB$ is $\omega$-bisectorial, i.e. $\sigma(DB)\subset S_\omega$ and there are resolvent bounds 
$\|(\lambda I - DB)^{-1}\| \lesssim 1/ \dist(\lambda, S_\omega)$ when $\lambda\notin S_\omega$.
\item[{\rm (ii)}]
The operator $DB$ has range $\ran(DB)=\ran(D)$ and null space $\nul(DB)$ such that topologically (but in general non-orthogonally) one has
$$
\mH = \clos{\ran(DB)} \oplus \nul(DB).
$$
\item[{\rm (iii)}] 
The restriction of $DB$ to $\clos{\ran(DB)}$ 
is a closed and injective operator with dense range in
$\clos{\ran(DB)}$, with estimates on spectrum and resolvents as in (i).
\item[{\rm (iv)}]
If $BD$ satisfies the quadratic estimate (\ref{eq:sfBD}), for example if (H1-6) are all satisfied, then for all $g\in \mH$,
\begin{equation}\left(\qe{t^kDB(I + t^{2k}DBDB)^{-1} g}\right)^{1/2}\lesssim \|g\|.
\label{eq:sfBD*}\end{equation} 
\end{itemize}
\end{prop}

If $B$ is strictly accretive on all $\mH$, then $DB= B^{-1}(BD)B$, so 
$DB$ and $BD$ are similar operators.
In this case, bisectoriality, resolvent bounds, quadratic estimates and boundedness of 
functional calculus carries over from immediately $BD$ to $DB$.

However we are only assuming that $B$ is strictly accretive on $\ran(D)$ as in (H5).
Denote by $B'$ the restricted operator
$B'=B:\clos{\ran(D)}\rightarrow B\clos{\ran(D)}= \clos{\ran(BD)}$.
By (H4,5) this is an isomorphism, and we have $DB=(B')^{-1}(BD)B'$ on $\clos{\ran(DB)}$. Thus we can transfer results about $BD$ on $\clos{\ran(BD)}$ to results about $DB$ on $\clos{\ran(DB)}
$.

To extend the action of $DB$ to all of $\mH$,   we use (ii).  To check (ii), 
note that
$DB$ and $B^*D$ are adjoint operators
so that $\clos{\ran(DB)}=\nul(B^*D)^\perp$ and $\nul(DB)= \ran(B^*D)^\perp$
and the splitting follows from Proposition~\ref{prop:typeomega}(ii) on taking orthogonal complements.

\section{Spectral decomposition and dependence on $B$}
For 
 this section,
we specialise to the case of first order differential operators $D$, i.e.~$k=1$, and assume that $B$ is a multiplication operator.

We continue to treat $DB$, though the following results are readily adapted to $BD$. 

\begin{prop}  \label{spectraldecomp} Suppose the hypotheses (H1-5) are satisfied with $k=1$, and that $B$ is multiplication by a function $B\in L^\infty(\R^n,\mL(\C^m))$. Then the following hold.  
\begin{itemize}
\item[(i)] The operator $\sgn(DB)$ is bounded on $\mH$.
\item[(ii)]   The operator $DB$ has a spectral decomposition $\clos{\ran(D)}=\clos{\ran(DB)}=\mH_{DB+}\oplus\mH_{DB-}$ where $\mH_{DB\pm}=\{v\in\clos{\ran(D)} : \sgn(DB)v=\pm v\}$.
 
\item[(iii)] The projections of $\clos{\ran(D)}$ onto $\mH_{DB\pm}$ are $E^B_\pm=1/2(I\pm \sgn(DB))$.
\item[(iv)] The restriction of $DB$ to $\mH_{DB+}$ is $\omega$-sectorial and thus generates a bounded analytic semigroup $e^{-tDB}$ in $\mH_{DB+}$ which satisfies $\lim_{t\to\infty}e^{-tDB}v\to0$ and $\lim_{t\to0}e^{-tDB}v\to v$ for all $v\in \mH_{DB+}$.
\end{itemize}
\end{prop}

Indeed $V:=e^{-t DB}v$ is the unique function in $C^1(\R_+,\mH)$
  which satisfies \linebreak
$\frac\partial{\partial t}V(t)+DBV(t)=0$ for $t>0$, as well as the limiting conditions just mentioned.

 Part (i) is proved in the same way as Proposition \ref{prop:firstorder} once we have the quadratic estimates (\ref{eq:sfBD}) and (\ref{eq:sfBD*}). The other parts follow as a consequence.

We remark that, by the preceding two propositions, there is a spectral decomposition 
$$
\mH = \mH_{DB+}\oplus\mH_{DB-} \oplus \nul(DB)
$$
corresponding to the right and left sectors of $S_\omega$ and $\{0\}$.

It is an important consequence of proving bounds for a general class of matrices $B$, that all of the corresponding bounded operators such as $\sgn(DB)$ depend analytically on $B$. To prove this, start by showing analytic dependence for the resolvents, and then prove analyticity for more general operators by using the fact that uniform limits of analytic functions are analytic.  
See \cite{AKMc, AAM} for further details. We need this fact for the projections $E^B_\pm$ defined above.

\begin{prop} Suppose that $z\mapsto B_z$ is an analytic function from $z$ in an open set $\Omega\subset\C$ to functions $B_z\in L^\infty(\R^n,\mL(\C^m))$, and that the operators $D$ and multiplication by $B_z$ in $\mH$ satisfy (H1-5) with $k=1$ under uniform bounds. Then the projections $E^{B_z}_\pm:\clos{\ran(D)}\to\mH_{D{B_z}\pm}\subset\clos{\ran(D)}$ depend analytically on $B$. \label{analytic}
\end{prop}

\section{Elliptic systems with square integrable boundary data}\label{ESSIBD}

Let us illustrate the power of our estimates by showing that elliptic systems on $\R^{1+n}_+=\{(t,x);t>0\}$ of the form
\begin{equation}\divv_{t,x} A({x})\nabla_{t,x} F(t,x)=0\label{pde}
\end{equation}
are well-posed under  $L^2$ Neumann data on $\R^n$, when the coefficient matrix $A$ is self-adjoint, strictly accretive, and has coefficients which are bounded measurable functions of $x\in\R^n$. The functions $F$ map $\R^{1+n}_+$ to $\C^N$.  One can also handle  $L^2$ Dirichlet and $\dot H^1$ Dirichlet data, thus  generalising  results   concerning the case $N=1$ and real symmetric coefficients \cite{JK1,DJK,KP}.  See \cite{AAM} for a more extensive 
theory and for the historical background. 

\subsection{Results}

On writing $$U(t,x)=\left[\begin{array}{c}U_0(t,x)\\{\bf U}(t,x)\end{array}\right]=\left[\begin{array}{c}\tfrac{\partial}{\partial t}F(t,x)\\ \nabla_xF(t,x)\end{array}\right]=\nabla_{t,x} F(t,x)$$ 
where $U$ takes values in $\C^{N}\oplus\C^{nN}$, the second order equation can be rewritten as a first order system
\begin{align}
\divv_{t,x} A({x})U(t,x)&=0\label{system1}\qquad\text{and}\\
\curl_{t,x} U(t,x)&=0\label{system2}
\end{align}
for $(t,{x}) \in \R^{n+1}_+$. We assume throughout this section that the
 coefficient matrix $A\in L^\infty(\R^n, \mL(\C^N\oplus\C^{nN}))$ is strictly accretive in the sense that $$\re\left(A\left[\begin{array}{c}f \\ \nabla_xg \end{array}\right],
\left[\begin{array}{c}f \\ \nabla_xg \end{array}\right]\right)\geq\kappa(\|f\|^2+\|\nabla_x g\|^2)
$$
for some $\kappa>0$ and all $f\in L^2(\R^n,\C^N)$, $g\in \dot H^1(R^n,\C^N)$. For the moment we do not assume self-adjointness of $A$.

The {\it  Neumann problem} for (\ref{pde}) is {\it well-posed in the} $L^\infty(L^2)$ {\it sense} means that, given $w \in L^2(\R^n, \C^N)$, there exists a unique function $U\in C^1(\R_+,L^2(\R^n,\C^{N}\oplus\C^{nN}))$ which satisfies (\ref{system1}) and (\ref{system2}) on $\R^{1+n}_+$, as well as $\lim_{t\to\infty}U(t,.)=0$ and $\lim_{t\to 0}U(t,.)=u$ in the $L^2$ sense, where 
\begin{equation}(Au)_0=\sum_{j=0}^n A_{0,j} u_j =w\qquad\text{on} \ \R^n\ .\label{neu}\end{equation}

Our aim is to prove the following result.

\begin{prop} Assume, in addition to the above conditions, that $A$ is self-adjoint. Then the Neumann problem for (\ref{pde}) is  well-posed in the $L^\infty(L^2)$ sense.\label{prop:nsa}
\end{prop}

In fact the solutions satisfy quadratic estimates 
and have non-tangential maximal function estimates.  See  \cite{AAM} for this, and for a treatment of more general conditions on $A$ and other boundary conditions.

\subsection{A related equation} 
In what follows, we write 
\begin{equation} A=\left[\begin{array}{cc}
a & b \\ c & d \end{array}\right], \end{equation}
where $a,b,c,d$ are $L^\infty$ functions taking values $a({ x})\in \mL(\C^N), b(x)\in\mL(\C^{nN},\C^{N}), c(x)\in\mL(\C^{N},\C^{nN}), d(x)\in\mL(\C^{nN})$ for a.a. $x\in\R^n$. Also write
\begin{equation}\ U=\left[\begin{array}{c}
U_0 \\ {\bf U} \end{array}\right]\ ;\ V=\left[\begin{array}{c}
V_0 \\ {\bf V} \end{array}\right]=\left[\begin{array}{cc}
a&b \\ 0&1 \end{array}\right]U
\end{equation}
where $U_0,V_0:\R^{1+n}_+\to \C^N$ and ${\bf U},{\bf V}:\R^{1+n}_+\to \C^{nN}$
and set
\begin{equation} D=\left[\begin{array}{cc}
0&\divv \\ -\nabla &0 \end{array}\right]
\end{equation}
where $\nabla=\nabla_x$ and $\divv=-\nabla^*$. This operator $D$ satisfies (H1,2,3).

Note that $\clos{\ran(D)}= L^2(\R^n,\C^N)\oplus\clos{\ran(\nabla)}$. So the coercivity condition on $A$ is really coercivity on $\clos{\ran(D)}$, i.e.
$$\re(ADu,Du)\geq\kappa\|Du\|^2$$ for all $u\in\dom(D)$. 

 We note that this inequality implies the pointwise accretivity $\re(a(x)\xi,\xi) \geq  \kappa |\xi|^2$
for all $\xi \in \C^N$ so we may define
$$  \hat A=\left[\begin{array}{cc}
1 & 0 \\ c & d \end{array}\right]\left[\begin{array}{cc}
a&b \\ 0&1 \end{array}\right]^{-1}.
$$

In the next subsection we show that $\hat A$ is also bounded and strictly accretive on $\clos{\ran(D)}$. 

\begin{prop} A function $U\in C^1(\R_+,L^2(\R^n,\C^{N}\oplus\C^{nN}))$ satisfies (\ref{system1}) and (\ref{system2}) on $\R^{1+n}_+$ if and only if the function $V\in C^1(\R_+,L^2(\R^n,\C^{N}\oplus\C^{nN}))$ satisfies 
\begin{align}
\frac\partial{\partial t}V+D{\hat A}V&=0\label{system3}\qquad\text{and}\\
  \curl_x {\bf V}(t, .)&=0\label{system4}
\end{align}
 on $\R^{1+n}_+$.
\end{prop}

\begin{proof} Equation (\ref{system3}), namely
$$\frac\partial{\partial t}\left[\begin{array}{c}
V_0 \\ {\bf V} \end{array}\right] + \left[\begin{array}{cc}
0&\text{ div}\, \\ -\nabla &0 \end{array}\right]\left[\begin{array}{cc}
1 & 0 \\ c & d \end{array}\right]\left[\begin{array}{cc}
a&b \\ 0&1 \end{array}\right]^{-1}\left[\begin{array}{c}
V_0 \\ {\bf V} \end{array}\right] = 0\ ,
$$ 
is equivalent to
$$\frac\partial{\partial t}\left[\begin{array}{cc}
a&b \\ 0&1 \end{array}\right]\left[\begin{array}{c}
U_0 \\ {\bf U} \end{array}\right] + \left[\begin{array}{cc}
0&\divv \\ -\nabla &0 \end{array}\right]\left[\begin{array}{cc}
1 & 0 \\ c & d \end{array}\right]\left[\begin{array}{c}
U_0 \\ {\bf U} \end{array}\right] = 0
\ ,$$
which is in turn equivalent to the pair of equations
\begin{align*}\left[\begin{array}{cc} \frac\partial{\partial t}  &\divv\end{array}\right]
\left[\begin{array}{cc}
a & b \\ c & d \end{array}\right]\left[\begin{array}{c}
U_0 \\ {\bf U} \end{array}\right] &= 0\\
\frac\partial{\partial t}{\bf U}-\nabla U_0&=0\ .
\end{align*}
The first of these is (\ref{system1}), while the second, when taken together with (\ref{system4}), is (\ref{system2}). 
\end{proof}

The Neumann condition (\ref{neu}) becomes $V_0=w$ on $\R^n$.

\subsection{The map $A\mapsto\hat A$}
Let us look at  the correspondence between a bounded strictly accretive matrix  $A=\left[\begin{array}{cc}
a & b \\ c & d \end{array}\right]$ and $\hat A=\left[\begin{array}{cc}
1 & 0 \\ c & d \end{array}\right]\left[\begin{array}{cc}
a&b \\ 0&1 \end{array}\right]^{-1}$. 
 As observed before,  the sub-matrix $a$ is itself bounded, strictly accretive on $L^2(\R^n,\C^N)$, and  
the inverse $\left[\begin{array}{cc}
a&b \\ 0&1 \end{array}\right]^{-1}$ exists and is bounded. For interest, we note that $\hat{\hat A}=A$.

\begin{prop} \label{prop:hatA} Given a bounded matrix $A$ as above which is strictly accretive on  $\clos{\ran(D)}$, then $\hat A$ is also bounded, strictly accretive on $\clos{\ran(D)}$.
\end{prop} 

\begin{proof} This is easily verified, once we have the following identity:
\begin{align*} \re\left(\hat A\left[\begin{array}{cc}
a & b \\ 0 & 1 \end{array}\right]\left[\begin{array}{c}
u_0 \\ {\bf u} \end{array}\right], \left[\begin{array}{cc}
a & b \\ 0 & 1 \end{array}\right]\left[\begin{array}{c}
u_0 \\ {\bf u} \end{array}\right]\right)&=\re((u_0,au_0+b{\bf u}) + (cu_0+d{\bf u}, {\bf u}))\\&=\re((au_0+b{\bf u},u_0) + (cu_0+d{\bf u}, {\bf u}))\\&=
 \re\left(A\left[\begin{array}{c}
u_0 \\ {\bf u} \end{array}\right], \left[\begin{array}{c}
u_0 \\ {\bf u} \end{array}\right]\right)\ .
\end{align*}
\end{proof}

Self-adjointness is not preserved under this transformation. Indeed, it can readily be checked that $A$ is self-adjoint if and only if $\hat A$ has the form $\hat A =
\left[\begin{array}{cc}\alpha&\beta\\ \gamma&\delta\end{array}\right]$ with $\alpha=\alpha^*$, $\beta=-\gamma^*$ and $\delta=\delta^*$.

\subsection{Results for $D\hat A$}\label{related}

   We now apply the theory which we have developed concerning operators of the form $DB$ to the present situation. So take $\mH=L^2(\R^n,\C^{N}\oplus\C^{nN})$, $B=\hat A$ and 
\begin{equation*}D=\left[\begin{array}{cc}
0&\divv \\ -\nabla &0 \end{array}\right] .
\end{equation*}

The pair  ($D$, $\hat A$) satisfies (H1-5) and $\hat A$ is a multiplication operator.  Therefore, by Proposition \ref{spectraldecomp}, the spectral projection $E^{\hat A}_+:\clos{\ran(D)}\to \mH_{D\hat A+}$ is bounded. 

The functions $v\in\mH_{D\hat A+}$ are precisely the boundary values of functions $V$ on $\R^{1+n}_+$ which satisfy (\ref{system3}) and (\ref{system4}). Moreover $V(t)=e^{-tD\hat A}v$ for $t>0$. And the Neumann problem for (\ref{pde}) is the determination of $v$, and hence $V$, from its first component,  $v_0=w$, where  $w\in L^2(\R^n,\C^N)$ is given. We are writing  $v=\left[\begin{array}{c} v_0\\ {\bf v}\end{array}\right]$. So, on defining the operator $P_{\hat A}: \mH_{D\hat A+}\to L^2(\R^n,\C^N)$ by $P_{\hat A} v =v_0$, we have the following result.

\begin{prop} The Neumann problem (\ref{pde})  is  well posed in the $L^\infty(L^2)$ sense if and only if  $P_{\hat A}:\mH_{D\hat A+}\to L^2(\R^n,\C^N)$ has a bounded inverse. 
\end{prop}

This means that, given $w\in L^2(\R^n,\C^N)$, there exists a unique function $v\in \mH_{D\hat A+}$ with $v_0=w$, and hence there exists a unique function $V\in C^1(\R_+,L^2(\R^n,\C^N\oplus\C^{nN}))$ which satisfies
$\frac\partial{\partial t}V+D\hat AV=0$, $\lim_{t\to\infty}V(t)=0$, $\lim_{t\to0}V(t)=v$.

Now the Neumann problem is not always well-posed \cite{KR}, and $P_{\hat A}$ is not always an isomorphism.

Our aim though is to prove Proposition \ref{prop:nsa}, or in other words, to show that when $A$ is self-adjoint, then  $P_{\hat A}$ is an isomorphism.

\subsection{Unperturbed operators}\label{unperturbed} In order to tie our results in with the classical theory, we start with $A=\hat A=I$. In this case $v=\left[ \begin{array}{c}
v_0  \\ {\bf v}  
\end{array}
\right]\in \mH_{D+}$ if and only if $\curl {\bf v}=0$ and $Dv=\sqrt{(D^2)}v$, i.e.
$$
\left[ \begin{array}{cc}
0&\divv  \\ -\nabla&0  
\end{array}
\right]
\left[ \begin{array}{c}
v_0  \\ {\bf v}  
\end{array}
\right]
=
\left[ \begin{array}{cc}
\sqrt{-\Delta}  &0\\ 0&\sqrt{-\nabla \divv}
\end{array}
\right]
\left[ \begin{array}{c}
v_0  \\ {\bf v}  
\end{array}
\right]$$
where $\Delta=\divv\nabla$ is the Laplacian on $\R^n$. That is, ${\bf v} = -\nabla (-\Delta)^{-1/2} v_0$ , or equivalently, 
$v_0=(-\Delta)^{-1/2}\divv\, {\bf v}$ and $\curl\, {\bf v}=0$. 

Thus the map $P_I$ is an isomorphism, and so, as we know,  the Neumann problem for the Laplacian on the upper half space is well posed in the $L^\infty(L^2)$  sense.

\subsection{Rellich inequality}\label{rellich} We now consider bounded, strictly accretive, self-adjoint matrices $A$. In this case $\hat A$ has the form $\hat A=
\left[
\begin{array}{cc}
\alpha   & \beta  \\ \gamma  &\delta\end{array} \right]$ with  $\alpha =\alpha^*$, $\beta=-\gamma^*$ and $\delta = \delta^*$, or in other words,
$(\hat A)^*K=K\hat A$ where $K=\left[
\begin{array}{cc}
1   & 0  \\0  &-1\end{array} \right]$. 

Since $KD+DK=0$, we have, for $v\in \mH_{D\hat A+}$ and $V=e^{-tD\hat A}v$, that
\begin{equation*}
(K\hat Av,v)=-\int_0^\infty(K\tfrac\partial{\partial t}V, \hat AV)+(K\hat AV, \tfrac\partial{\partial t}V)
=\int_0^\infty(KD\hat AV,\hat AV)+(DK\hat AV,\hat AV)=0\end{equation*}
or in other words,
$$(\alpha v_0,v_0)+2\re(\beta {\bf v},v_0) - (\delta{\bf v},{\bf v})=0.$$
Therefore, as ${\bf v}\in\clos{\ran(\nabla)}$ and $\delta$ is strictly accretive on $\clos{\ran(\nabla)}$, we have
$$\|{\bf v}\|^2\lesssim \re(\delta {\bf v},{\bf v}) \lesssim\|v_0\|^2+\|{\bf v}\|\|v_0\|$$
and hence
 the Rellich inequality
 $$ \|{\bf v}\|\lesssim \|v_0\|$$
for all $v\in \mH_{D\hat A+}$. This in turn tells us that the mapping $P_{\hat A}:v\to v_0$ satisfies the \textit{a priori} estimates $\|P_{\hat A} v\|\approx\|v\|$ for all $v\in \mH_{D\hat A+}$. 
 
 To prove surjectivity of $P_{\hat A}$, use the method of continuity (i.e. the invariance of semi-Fredholm index) for the family $B_\tau = \tau \hat A +(I-\tau)I$, $0\leq\tau\leq1$, and the fact that $P_I$ is an isomorphism. For this we  need to know that the spaces $\mH_{D{B_\tau}+}$, or in other words the projections $E^{B_\tau}_+$, depend continuously on $\tau$. 
 
 In fact $B_\tau$ depends analytically on $\tau\in\C$. Moreover, on some open subset $\Omega\subset\C$ which contains the closed real interval $[0,1]$, the operators $D,B_\tau$ satisfy (H1-5) with uniform bounds. Therefore, by Proposition \ref{analytic},  the projections $E^{B_\tau}_+$ are actually analytic in $\tau$.
 
This completes the proof of Proposition \ref{prop:nsa}.

\begin{rem} 

The results on analytic dependence imply even more, namely that the Neumann problem is well-posed in the $L^\infty(L^2)$ sense for all coefficient matrices $A$ which are sufficiently small perturbations of a self-adjoint matrix. See \cite{AAM}.

Our aim has been to show that the
functional calculus provides a mechanism to turn a Rellich estimate
(i.e., comparability in the $L^2$ norm of transverse and tangential derivatives,
which always holds for self-adjoint, $t$-independent, divergence
form elliptic operators) into $L^2$ solvability. Previously, such a mechanism
had existed in the case of a single equation with real
symmetric coefficients (e.g., as in \cite{JK1,KP}), or in the
case of a block diagonal matrix (i.e., the setting of the Kato problem,
in which case the square function estimates are equivalent to a
Rellich estimate), {or in the case of a system with constant coefficients}.
\end{rem}

\section{Inhomogeneous $D$ when $k=1$}

Consider  operators $D$ and $B$ on $\mH=L^2(\R^n, \C^m)$ with the requirement (inhH) which consists of   (inhH1):   The operator $D:\dom(D) \rightarrow \mH$ is a (inhomogeneous)  first order differential operator with constant coefficients, (H2),   (inhH3): 
$$
\| u\| +  \|\nabla  u\| \lesssim \|Du\|  , \quad  \text{for all }\ u \in \dom(D) \cap \clos{\ran(D)},
$$
 (H4,5)  and $B$ is a multiplication by a matrix-valued function. 

Then we claim that the conclusion of the main theorem is valid replacing (H) by (inhH).  The proof consists in going back to the homogeneous case   (It would be nice to have a direct proof as in Section \ref{sec:PMT}). 

Write $D= D_1 + D_0$ where $D_1$ is a homogeneous first order differential operator and $D_0$ is multiplication by a constant matrix. We observe that since $D$ is self-adjoint, so are $D_0$ and $D_1$ as seen by computing the Fourier symbols.
Also, using Fourier arguments again,  one can check that (H3)  holds for  $D_1$.
So one can define $(I+itBD_1)^{-1}$, the corresponding $Q_t^1$, and obtain
the conclusion of Theorem \ref{th:main} for $Q_t^1$.
Observe that 
$$
\|(I+itBD)^{-1} - (I+itBD_1)^{-1}\|=\| (I+itBD)^{-1} (tBD_0) (I+itBD_1)^{-1}\|  \lesssim | t |.
$$
 Recall that $Q^B_t= \frac 1 {2i} (R^B_{-t} - R^B_t)$, hence
 $$
 \|tBD(I+t^2BDBD)^{-1} - tBD_1(I +t^2BD_1BD_1)^{-1} \| \lesssim t.
$$
So $\qezero{\Theta_tu}{1} \lesssim \|u\|^2$ for all $u \in \mH$ follows from \eqref{eq:sfBD} for 
$BD_1$. 

It remains to control $\qeinf{\Theta_tu}1$  and we know that it suffices to assume $u=Dv$
with $v\in \dom(D)$. We can also assume $v\in \clos{\ran(D)}$ from \eqref{eq:splitting}. But $\|\Theta_t D\| \lesssim t^{-1}$, hence  $\qeinf{\Theta_tDv}1 \lesssim \|v\|^2$  and we conclude invoking (inhH3).

For the  application to the Kato problem for second order elliptic operators with lower order terms 
$L= - \divv A \nabla + \divv b - c\nabla +d$,  where $A, b, c, d$ are multiplication operators with $L^\infty(\R^n, \C)$ coefficients and $ \begin{bmatrix}
        A & b\\
      c & d
     \end{bmatrix}$ is strictly accretive on $L^2(\R^n, \C^{n+1})$,  we take
\begin{equation}  \label{eq:definhD}
 D:=
     \begin{bmatrix}
        0 & \divv_x & 1 \\
        -\nabla_x & 0 &0 \\
        1 & 0& 0
     \end{bmatrix}, \quad 
        B:=  \begin{bmatrix}
        1 & 0 & 0 \\
        0 & A & b\\
        0& c & d
     \end{bmatrix}.
\end{equation}

 \begin{rem} The inhomogenous version for higher order $D$  is  in \cite[Section 0.7]{AT}. Further results on inhomogeneous problems are in \cite{AKMc1}.
 \end{rem}

\section{Related operators and  further  applications 
} \label{sec:relops}

 The aim of this section is to see that operators $BD$ or $DB$ are the building blocks of other operators obtained by functional analytic considerations in such a way that results for $DB$ and $BD$ apply immediately. We will finish with a concrete application to the functional calculus of Dirac type operators on forms.

\subsection{Operators of type $\Gamma+ B^{-1}\Gamma^* B$}

This class of perturbed Dirac operators was studied in \cite{AKMc}, where bisectoriality
and quadratic estimates where proved. As a corollary, our Theorem~\ref{th:main} was
proved in \cite[Theorem 3.1(iii)]{AKMc}.
In this section we prove the converse to this, i.e. we deduce the results in \cite{AKMc} from
Theorem~\ref{th:main}.

As in  \cite{AKMc}, we consider the following slightly more general situation.
Let $\Gamma$ and $\Gamma^*$ be adjoint nilpotent operators in an  arbitrary   Hilbert space $\mH$, i.e. $\Gamma^2= (\Gamma^*)^2=0$,
let $B_1, B_2$ be bounded operators such that $B_1$ is strictly accretive on $\ran(\Gamma^*)$
and $B_2$  is strictly accretive on $\ran(\Gamma)$.
Furthermore, assume  that
$\Gamma^*B_2B_1\Gamma^*=0$ and $\Gamma B_1B_2\Gamma=0$.
This holds in particular if $(B_1,B_2)=(B^{-1},B)$  with $B$ strictly accretive on all of $\mH$ .
For more details on this hypothesis, we refer to (H1-3) in \cite{AKMc}. 
In this case, the operator   on $\mH$ 
$$
\Pi_B:= \Gamma+\Gamma^*_B,\qquad \Gamma^*_B:= B_1 \Gamma^* B_2
$$
induces a Hodge type splitting
\begin{equation}  \label{eq:hodge}
  \mH=  (\nul(\Gamma^*_B)\cap\nul(\Gamma))\oplus \clos{\ran(\Gamma^*_B)}\oplus \clos{\ran(\Gamma)}. 
\end{equation}

Now consider the operators 
$$
  D:= \begin{bmatrix} 0 & \Gamma^* \\ \Gamma & 0 \end{bmatrix},
  \qquad
 B:= \begin{bmatrix} B_1 & 0 \\ 0 & B_2 \end{bmatrix}
$$
acting in $\mH\oplus \mH$.
Note that the hypothesis on $\Gamma, \Gamma^*, B_1, B_2$ above is equivalent
with (H2,4,5) for $D, B$.
 Using the fact that the Hodge splitting (\ref{eq:hodge}) is topological,   the map 
$$
S:\mH \longrightarrow \mH \oplus \mH
: f_0+f_{1}+f_2\longmapsto 
  \begin{bmatrix} f_1 \\ B_2 f_2 \end{bmatrix}
$$
 is well-defined and bounded,  and one can show that  the restriction of $S$ to  $\clos{\ran(\Pi_B)}= \clos{\ran(\Gamma^*_B)}\oplus \clos{\ran(\Gamma)}$  is an isomorphism onto $\clos{\ran(BD)}$. 
Since 
\begin{multline*}
  BDS(0+f_1+f_2)= 
  \begin{bmatrix} 0 & B_1\Gamma^* \\ B_2\Gamma & 0 \end{bmatrix}
  \begin{bmatrix} f_1 \\ B_2 f_2 \end{bmatrix}
  = \begin{bmatrix} \Gamma^*_B f_2 \\ B_2 (\Gamma f_1) \end{bmatrix} \\
  = S(0+ \Gamma^*_B f_2 + \Gamma f_1)= S(\Pi_B(0+ f_1+f_2) ),
\end{multline*}
we have shown  that the restrictions  $BD:\clos{\ran(BD)}\rightarrow \clos{\ran(BD)}$ and 
$
  \Pi_B: \clos{\ran(\Pi_B)}\longrightarrow \clos{\ran(\Pi_B)}
$
are similar operators.  Consequently, we obtain the following.

\begin{prop}
  Let $\Pi_B$ in $\mH$ and $BD$ in $\mH\oplus \mH$ be as above.
$ \Pi_B$ 
is a bisectorial operator on $\mH$ with resolvent bounds.
Furthermore, if $BD$
satisfies quadratic estimates, so does
$\Pi_B$.
\end{prop}

\subsection{Operators of type $BD_1+D_2B^{-1}$}

In this section we aim to deduce quadratic estimates  for operators well adapted to boundary value problems for differential forms. 
Similar to our discussion of the class of operators $\Pi_B$ above, we consider the
following slightly more general class of operators 
$$
B_1D_1+ D_2 B_2
$$ acting in an  arbitrary   Hilbert space $\mH$.
We assume that the two pairs of operators $(D_1,B_1)$ and $(D_2,B_2)$ both
satisfy (H2,4,5), and the following compatibility conditions.
For the unperturbed operators $D_1$ and $D_2$, we assume that
$\ran(D_1)\subset \nul(D_2)$ and $\ran(D_2)\subset \nul(D_1)$.
This means in particular that $D_2D_1=0= D_1D_2$ on appropriate domains.
For the perturbed operators we assume that 
 $B_2B_1: \ran(D_1)\rightarrow \nul(D_2)$ so that $D_2B_2B_1D_1=0$
on  $\dom(D_{1})$.  This holds in particular if $(B_1,B_2)=(B,B^{-1})$  with $B$ strictly accretive on all of $\mH$. 

Similar to the fact that an operator $\Pi_B$ acts by swapping the two ranges in its Hodge splitting,
an operator $B_1D_1+ D_2 B_2$ acts diagonally in the associated splitting
\begin{equation}  \label{eq:hodgetypesplitting}  
  \mH=  \Big(\nul(B_1D_1) \cap \nul(D_2B_2)\Big)\oplus \clos{\ran(B_1D_1)} \oplus \clos{\ran(D_2B_2)} 
\end{equation}
of $\mH$. To see this splitting, note that
$$
  \mH= \clos{\ran(B_1D_1)} \oplus \nul(B_1D_1) = \nul(D_2B_2)\oplus \clos{\ran(D_2B_2)},
$$
and (\ref{eq:hodgetypesplitting}) follows  by intersecting these two splittings since
$\ran(B_1D_1)\subset\nul(D_2B_2)$ and $\ran(D_2B_2)\subset \nul(B_1D_1)$.
Since $B_1D_1$ and $D_2B_2$ act as
$$
  B_1D_1= 0\oplus B_1D_1\oplus  0,\qquad D_2B_2=0\oplus 0 \oplus D_2B_2 
$$
in the splitting (\ref{eq:hodgetypesplitting}), 
the operator $B_1D_1+ D_2B_2$ is the direct sum of these two operators,
namely
$$
   B_1D_1+D_2B_2:\mH\longrightarrow \mH: f_0+f_1+f_2\longmapsto 0+ B_1D_1 f_1+D_2B_2 f_2. 
$$

This shows the following.

\begin{prop}  \label{prop:dbtobdplusdb}
  Assume that $D_1,B_1$ and $D_2, B_2$ are as above. Then $B_1D_1+D_2B_2$ is a bisectorial operator on $\mH$ with resolvent bounds. 
Furthermore, if $B_1D_1
$
and 
$D_2B_2$
satisfy
quadratic estimates, then so does
 $B_1D_1+D_2B_2$.

\end{prop}

 \subsection{ An application to Dirac type equations }

We end with an application of the quadratic estimates for the class of operators 
$B_1D_1+D_2B_2$ to boundary value problems for differential forms. 
The goal is to prove that the underlying operator $T_B$ used in \cite{AAH}
to obtain perturbation results for BVPs for Dirac type equations, and in particular for
BVPs for second order divergence form equations, has quadratic estimates
for {\bf all} complex strictly accretive coefficients $B$.
In \cite{AAH}, quadratic estimates were proved only for special types of coefficients $B$,
namely for small perturbations of real symmetric, constant and block form coefficients.

The operators $T_B$ are infinitesimal generators for the studied Dirac equations,
and following \cite[Definition 2.10]{AAH} we have
$$
  T_B= -iM_B^{-1}(\ud+ B^{-1}\ud^* B),
$$
which acts in $L^2(\R^n;\wedge)$, where $\wedge= \wedge_\C \R^{1+n}$ is the full complexified
exterior algebra of $\R^{1+n}$.
Here $\ud$ is a nilpotent differential operator, i.e. $\ud^2=0$, and $B$ is a strictly accretive multiplication operator
so that the operator $T_{B}$, modulo the factor $-iM_B$ (being an invertible, non-accretive, multiplication operator),
is of type $\Pi_B$.   More precisely, if $\mu f= e_0\wedg f$ is exterior multiplication by the basis vector normal to $\R^n$ and $m:= \mu+\mu^*$ then $\ud:= imd$ where $d$ is the exterior derivative, and
$M_B:= \mu^*\mu- B^{-1}\mu\mu^* B$.  
In the work \cite{AAH}, this factor complicated the application of results for the class $\Pi_B$ 
and only gave partial results. 
However, in connection with the later work \cite{AAM}, it was realized that the operators $T_B$
actually are similar to operators of type $B_1D_1+D_2B_2$.
 This similarity uses the  transform $B\mapsto \hat B:= \uB \oB^{-1}$ of strictly accretive matrices analogous to
Proposition \ref{prop:hatA}, on
splitting the space $L^2(\R^n;\wedge)$ into normal and tangential forms and writing the operators  $B, \uB, \oB$ 
as the matrices
$$
  B=
  \begin{bmatrix}
     B_{\no\no} & B_{\no\ta} \\ B_{\ta \no} & B_{\ta\ta} 
  \end{bmatrix}, \quad 
\uB=  \begin{bmatrix}
    I & 0 \\ B_{\ta \no} & B_{\ta\ta}   \end{bmatrix},
\quad
  \oB=  \begin{bmatrix}
    B_{\no\no} & B_{\no\ta} \\ 0 & I  \end{bmatrix}.
$$ 
 This is summarized in  the following new result.
\begin{cor}
Let $B \in L^\infty(\R^n;\mL(\wedge))$ be any complex  coefficient matrix function which is  strictly  accretive on $L^2(\R^n, \wedge)$, matrix function.
Then 
$$
T_B =\oB^{-1} \Big( D_1\hat B  + \hat B^{-1} D_2\Big)\oB,
$$
where 
the differential operators
are $D_1:=  \mu^* d- \mu d^*$ and $D_2:= \mu^* d^*- \mu d$.
In particular
$T_B$ is an injective $\omega$-bisectorial operator, $\omega$
being the angle of accretivity of $\hat B$, has resolvent bounds and satisfies
quadratic estimates in $L^2(\R^n;\wedge)$.
\end{cor}

\begin{proof}
  Since $D_1,D_2$ are first order differential operators, according to Theorem \ref{th:main}, 
  Proposition~\ref{prop:dbtobdplusdb} and Proposition~\ref{pseudoloc}, 
    it suffices to verify the similarity. 
  Multiplying the equation with $BM_B$, it suffices to show that
$$
  -i(B\ud+ \ud^* B)= (BM_B) (\oB^{-1}D_1 \uB+ \uB^{-1}D_2\oB).
$$
Identifying $e_0\wedg f_1+f_2\in L^2(\R^n;\wedge \R^{1+n})$ with
$[f_1,f_2]^t\in L^2(\R^n;\wedge\R^n)^2$, the   above  definitions 
give
$-i\ud= \begin{bmatrix} 0 & d \\- d & 0\end{bmatrix}$, 
$-i\ud^*= \begin{bmatrix} 0 & d^* \\ -d^* & 0\end{bmatrix}$, 
$B M_B= \begin{bmatrix} -B_{\no\no} & 0 \\ 0 & B_{\ta\ta} \end{bmatrix}$, 
$D_1=- \begin{bmatrix} 0 & d^* \\ d & 0\end{bmatrix}$
and 
$D_2= -\begin{bmatrix} 0 & d \\ d^* & 0\end{bmatrix}$.
The similarity is now straightforward to verify.
\end{proof}

We end with the remark that it is only the part $\hat T_B$ similar to $D_1\hat B$
of the full operator $T_B$ that is needed for the application to boundary value problems
for $k$-vector fields / $k$-forms. 
This application is described in \cite[Section 6]{AAM}.
The complementary part $\hut T_B$ similar to $\hat B^{-1}D_2$ will mix $k$-vector fields
of different order $k$, but has the advantage of making the full operator $T_B$ injective.

\bibliographystyle{acm}

\end{document}